\documentclass[12pt]{article}
\usepackage{amsmath,amssymb,amsthm,fullpage}
\usepackage[numbers]{natbib}

\newcommand{\Z}{{\mathbb Z}}

\newcommand{\R}{{\mathbb R}}

\usepackage{braket}

\newcommand{\parens}[1]{\left( #1 \right)}

\theoremstyle{plain}
\newtheorem{theorem}{Theorem}[section]
\newtheorem{corollary}[theorem]{Corollary}

\theoremstyle{remark}
\newtheorem*{remark}{Remark}

\theoremstyle{definition}
\newtheorem{definition}{Definition}[section]

\usepackage{color}
\definecolor{red}{rgb}{.8,0,0}
\definecolor{green}{rgb}{0,.7,0}
\definecolor{blue}{rgb}{0,0,.8}

\title{Bounds on curvature in regular graphs}
\author{Peter Ralli\thanks{School of Mathematics, Georgia Institute of Technology.  Supported in part by NSF grant DMS 1407657.}}
\date{}
\begin{document}
\maketitle

\begin{abstract}
We study the curvature-dimension inequality in regular graphs.  We develop techniques for calculating the curvature of such graphs, and we give characterizations of classes of graphs with positive, zero, and negative curvature.  Our main result is to compare the curvature-dimension inequality in these classes to the so-called Ollivier curvature.  A consequence of our results is that in the case that the graph contains no subgraph isomorphic to either $K_3$ or $K_{2,3}$ these curvatures usually have the same sign, and we characterize the exceptions.  
\end{abstract}

\section{Introduction}
Recently there have been several attempts to translate the well-understood concept of curvature from Riemannian geometry to discrete spaces.  In the continuous setting, the Bochner formula characterizes harmonic functions in terms of the curvature.  Based on this, Bakry and \'{E}mery \cite{BE} developed the Curvature-Dimension inequality, which has since been adapted to define curvature in a discrete setting \cite{KKRT15}.  There have recently been many results on the spectral and isoperimetric properties of graphs under a bound on the CD-curvature \cite{CLY14}\cite{KKRT15} with further references therein.

An alternate notion of discrete curvature is the Ollivier curvature, proposed by Y. Ollivier\cite{Oll07} and independently by Sammer \cite{Sam05} which has been further investigated at length, e.g. \cite{Oll09}.  This curvature compares the minimum-transport distance of balls on a curved space to that of balls in Euclidean space.

In this work we investigate the CD-curvature for regular graphs, and compare to the Ollivier curvature.  It is known that these notions are not in general equivalent, indeed, in this work we give examples of graphs for which the curvature is positive in one notion but not the other.  For graphs that are triangle-free and have no subgraph isomorphic to $K_{2,3}$, we develop rules for calculating both types of curvature, and observe that the signs of the curvatures are usually equivalent.  We also characterize the classes of graphs for which the signs are not equivalent.  We calculate the curvature for several examples of interest, including the graph of the state space of the random interchange process.

In Section 2, we provide the definitions of both the curvature-dimension inequality and of Ollivier's curvature.  In Sections 3 and 4, we investigate the curvature of regular graphs that have no subgraph isomorphic to either $K_3$.  In addition we prove a number of short results related to these notions of curvature.

\section{Preliminaries}  

\subsection{Curvature-Dimension inequality} The $CD$-inequality was introduced in \cite{BE85} and several variations have been studied.  Our definitions follow from \cite{KKRT15} and we use the non-normalized Laplacian matrix.  For a $d$-regular graph, the curvature differs from that explored by Chung, Lin, Yau \cite{CLY14} by a multiplicative factor of $d$; those definitions result from choosing a normalized Laplacian.

If $G = (V,E)$ is a locally finite graph and $f,g:V\to\R$, then

\[\Gamma (f,g)(x) = \frac{1}{2}\sum_{y\sim x}\parens{f(y)-f(x)}\parens{g(y)-g(x)},\]
\[\Gamma f(x) :=\Gamma(f,f)(x),\]
\[\Delta f(x) = \sum_{y\sim x} f(y)-f(x),\]
\[\text{and }\Gamma_2 f = \frac{1}{2} \Delta\Gamma f - \Gamma(f,\Delta f).\]

Let $f:V\to\R$ and $x\in V$. Because $\Delta f = \Delta (f+c)$ and $\Gamma f = \Gamma(f+c)$, we may (and will) assume that $f(x) = 0$.  In that case straightforward manipulation (as in \cite{KKRT15}) reveals a form for $\Gamma_2f(x)$ that often simplifies computation: 

\begin{align*}\label{gamma}
2\Gamma_2f(x) = &\frac{1}{2}\sum_{\substack{u\sim v\sim x\\ d(x,u) = 2}}\parens{f(u)-2f(v)}^2 + \parens{\sum_{v\sim x}f(v)}^2 + \sum_{v\sim x} \frac{4-d(x)-d(v)}{2}f^2(v)
\nonumber\\
 +&\sum_{\Delta(x,v,u)}\biggl[2\parens{f(v)-f(u)}^2 + \frac{1}{2}\parens{f^2(v) + f^2 (u)}\biggr]\,,
\end{align*}

where $\Delta(x,v,u)$ means that $x\sim v\sim u\sim v$.

\begin{definition} We say that $G$ satisfies the $CD(\rho,\infty)$ condition at $x$ iff for all $f:V\to\R$, $\Gamma_2f(x)\geq\rho \Gamma f(x)$.\end{definition}

\begin{remark} Whether $CD(\rho,\infty)$ is satisfied for $G$ at $x$ is a local property; in particular, it depends only on the structure of edges that are incident to at least one neighbor of $x$.  It is possible to remove all other edges without affecting the curvature at $x$.
\end{remark}

An obvious concern is to characterize the non-constant function $f$ that minimizes $\Gamma_2f(x)/\Gamma f(x)$.  It is understood how to calculate the $f(u)$ if $d(x,u) = 2$.  A simple optimization reveals that, holding $f(v)$ constant for all $v\sim x$, the value of $f(u)$ minimizing $\Gamma_2f(x)$ is 
\[f(u) = \frac{2\sum_{v: x\sim v\sim u} f(v)}{\#\{v:x\sim v\sim u\}}.\]

\subsection{Ollivier curvature}

The second form of discrete curvature that we consider is the so-called Ollivier curvature.

For probability measures $\mu,\nu$ on $V$, the $L_1$ Wasserstein (i.e., minimum-transport) distance is \begin{align*}W_1(\mu,\nu) = \min_m \int_{V\times V} d(x,y) dm(x,y),\end{align*} where the minimum is taken over all probability measures $m$ on $V\times V$ so that \begin{align*}\int_V m(x,y) d\nu(y) = \mu(x)\text{ and } \int_V m(x,y) d\mu(x) = \nu(y).\end{align*}  In plain words, we transport $\mu$ to $\nu$ by shifting a load of size $m(x,y)$ along an $(x,y)$-geodesic, and $W_1$ minimizes the average distance transported over all choices of transport function $m$. 

Let $G$ be a $d$-regular graph.  For $x\in V$, define a probability measure $\mu_x$ so that \begin{align*}\mu_x(v) = \begin{cases} 
      \tfrac{1}{2} & \text{if }v = x \\
      \frac{1}{2d} & \text{if }v \sim x \\
      0 & \text{otherwise.}
   \end{cases}\end{align*}

\begin{definition} If $x,y\in V$ and $x\sim y$, the curvature is $\kappa(x,y) = 1-W_1(\mu_x,\mu_y)$.
\end{definition}

\section{$K_{2,3}$ and triangle-free graphs}

In this section we compare the $CD$ curvature to Ollivier's curvature for regular graphs that contain no subgraphs isomorphic to either $K_3$ or $K_{2,3}$.

\begin{definition}If $x\in V$ and $y,z\sim x$, we say that $y$ and $z$ are \textit{linked} if there is a vertex $w\neq x$ so that $y\sim w\sim z$.  We write $
y\approx z$ if $y$ and $z$ are linked and $y\not\approx z$ if not.\end{definition}

\begin{definition} If $x,y\in V$ and $x\sim y$, the non-linking number $N_x(y)$ is the number of other neighbors of $x$ that are not linked to $y$: $|\{w\sim x:w\neq y, w\not\approx y\}|$.\end{definition}

\begin{theorem} Let $G$ be $d$-regular and have no subgraph isomorphic to either $K_3$ or $K_{2,3}$.  Let $N = \max_{y\sim x} N_x(y)$.

(i) If $N = 0$, then $G$ satisfies $CD(\rho,\infty)$ at $x$ iff $\rho \leq 2$ ($G$ is positively curved at $x$.)

(ii) If $N = 1$, then $G$ satisfies $CD(\rho,\infty)$ at $x$ iff $\rho \leq 0$. ($G$ is flat at $x$.)

(iii) If $N \geq 2$, then $G$ does not satisfy $CD(0,\infty)$ at $x$. ($G$ is negatively curved at $x$.)
\end{theorem}

\begin{proof}
Because $G$ has no subgraph isomorphic to $K_{2,3}$, if neighbors $y,z$ of $x$ are linked, then they are linked by a unique vertex $w$, and the link $w$ cannot be adjacent to any neighbor of $x$ other than $y$ and $z$.\\

(i) If $N=0$, then $G$ is locally isomorphic to the hypercube $\Omega_d$. That is, there is a local isomorphism $\phi:V(G)\to V(\Omega_d)$ with the property that $\Gamma_2 f(x) = \Gamma_2 f\circ\phi(x)$ and $\Gamma f(x) = \Gamma f\circ\phi(x)$ for all $f:V\to\R$.  It is well-known \cite{KKRT15} that $\Omega_d$ has positive curvature, likewise $G$ has positive curvature.\\

(ii) To show that $CD(\rho,\infty)$ fails at $x$ if $\rho > 0$, we describe a test-function $f$ for which $\Gamma_2 f (x) = 0$:

  Let $y,z$ be a pair of unlinked neighbors of $x$, set $f(x) = 0$, $f(y) = 1$, $f(z) = -1$.  For any other neighbor of $x$, set $f = 0$.  On the second-neighbors of $x$, set $f$ to be twice the average of $f$ evaluated on all neighbors of $x$ adjacent to that vertex, so as to minimize $\Gamma_2$.  Notice that $y$ will be linked to $d-2$ neighbors of $x$(all besides $z$ and itself) by $d-2$ distinct vertices, each with $f = 1$, and will have one neighbor that is not adjacent to any other neighbor of $x$, at which $f = 2$.  Likewise, $z$ will have $d-2$ neighbors with $f = -1$ and one neighbor with $f = -2$.  And, each other neighbor of $x$ will be adjacent to one vertex with $f = 1$ and one with $f = -1$: the vertices linking them to $y$ and $z$ respectively.

With this function it is straightforward to see that $2\Gamma_2 f(x) = 0$.\\

We now show that for any function $f$, $2\Gamma_2f(x) \geq 0$, and therefore $CD(\rho,\infty)$ holds for any $\rho \leq 0$:

Because there is a bijection between pairs of linked neighbors of $x$ and the linking vertices,

\begin{align*}
\frac{1}{2}\sum_{\substack{u\sim v\sim x\\ d(x,u) = 2}}\parens{f(u)-2f(v)}^2 \geq \frac{1}{2}\sum_{\substack{y,z\sim x\\ y\approx z}} \parens{f(y)-2f(w)}^2 + \parens{f(z)-2f(w)}^2 \geq \sum_{\substack{y,z\sim x\\ y\approx z}} \parens{f(y)-f(z)}^2,
\end{align*} where $w$ is the vertex linking $y$ and $z$, and the sum is taken over all unordered pairs $y,z$ of linked neighbors of $x$.

Using this in the expression for $\Gamma_2$, we find

\begin{align*} 2\Gamma_2f(x) &\geq \sum_{\substack{y,z\sim x\\ y\approx z}} \parens{f(y)-f(z)}^2+\parens{\sum_{y\sim x}f(v)}^2 + (2-d)\sum_{y\sim x} f^2(v) \\&= \sum_{y\sim x} (3-d + \#\{z\sim x:y\approx z\})f^2(y) + \sum_{\substack{y,z\sim x\\ y\approx z}}2f(y)f(z) \\ &\geq \sum_{y\sim x} f^2(y) + \sum_{\substack{y,z\sim x\\ y\approx z}}2f(y)f(z) \geq \sum_{\substack{y,z\sim x\\ y\approx z}}\parens{f(y)-f(z)}^2 \geq 0,
\end{align*} because each neighbor $y$ of $x$ is linked to at least $d-2$ of the $d-1$ other neighbors of $x$, and not linked to at most $1$ other neighbor of $x$.\\

(iii).  Let $y$ be a neighbor of $x$ with $N_x(y)>1$.  Define a test-function $f$: $f(x) = 0$, $f(y) = d-1$, otherwise if $z\sim x$, $f(z) = -1$, and choose the values of $f$ on the second-neighbors of $x$ to minimize $\Gamma_2$.

It is straightforward to calculate that \begin{align*}
2\Gamma_2f(x) &= \biggl(\sum_{\substack{z\sim x\\ y\approx z}} \parens{f(y)-f(z)}^2\biggr) + (2-d)[(d-1)^2 + (d-1)]\\ &\leq (d-3)d^2 + (2-d)(d^2 - d) = -2d < 0,
\end{align*} and therefore $CD(0,\infty)$ fails at $x$.  The bound on the first term is due to the fact that $y$ is unlinked to at least $2$ of the $d-1$ other neighbors of $x$.

\end{proof}

We give a similar characterization of the Ollivier curvature:

\begin{theorem} Let $G$ be $d$-regular and have no subgraph isomorphic to either $K_3$ or $K_{2,3}$.  Let $N = \max_{y\sim x} N_x(y)$.

(i) If $N = 0$, then $\kappa(x,y) > 0$.

(ii) If $N = 1$, then $\kappa(x,y) \geq 0$.

(iii) If $N \geq 2$, then $\kappa(x,y) \leq 0$.
\end{theorem}

\begin{proof}
(i) As in our previous result $G$ must be locally isomorphic to $\Omega_d$, it is well-known that $\kappa(x,y) = \tfrac{1}{d}$ in this case.\\

(ii) Let $y,z$ be an unlinked pair of neighbors of $x$, and let $v_1,\dots, v_{d-2}$ be the other neighbors of $x$.  For each $v_i$ there is a $w_i$ that links $v_i$ and $y$, let $u$ be the neighbor of $y$ that is not in $\{x,w_1,\dots, w_{d-2}\}$.  $d(u,z)\leq 3$ as $u\sim y\sim x\sim z$.  To bound $W_1(\mu_x,\mu_y)$ from above, consider the transfer of mass along paths $x\to y, v_i \to w_i$ and $z\to u$.  This gives a bound of $W_1(\mu_x,\mu_y) \leq 1$, and so $\kappa(x,y)\geq 0$.\\

(iii) Let $y$ be a neighbor of $x$ with $N_x(y)>1$.  We use a test-function $f$ to bound the dual formulation of $W_1$,\begin{align}\label{wdual} W_1(\mu,\nu) = \max_{f\in Lip(1)} \int f\ d\mu-\int f\ d\nu.\end{align}  For a test-function, set $f(x) = 0$, $f(y) = 1$, $f(v) = 0$ if $v$ is any other neighbor of $x$, $f(w) = 1$ if $w$ links $y$ to a neighbor of $x$, $f(u) = 2$ if $u$ is any other neighbor of $y$, and $f = 1$ on every other vertex, so that $f$ is $1$-lipschitz.

\begin{align*} W_1(\mu_x,\mu_y)\geq \int f\ d\mu_y -\int f\ d\mu_x = 1 + \frac{N-2}{2d}\geq 1,\end{align*} so $\kappa(x,y) \leq 0$.
\end{proof}

\begin{remark} $\kappa > 0$ is satisfied under hypothesis (ii) for $G = C_5$, or graphs with similar structure for which (in the language of the proof) $d(u,z) = 2$.  $\kappa = 0$ is satisfied under hypothesis (iii) e.g. for the dodecahedral graph, again because $d(u,z) = 2$.  \end{remark}

We combine the previous results to show a relationship between the curvatures.

\begin{corollary} Let $G$ be $d$-regular and have no subgraph isomorphic to either $K_3$ or $K_{2,3}$.  Let $x\in V$.\\
(i) If and only if there exists $\rho > 0$ so that $G$ satisfies $CD(\rho,\infty)$ at $x$, then $\kappa(x,y)>0$ for all $y\sim x$.\\
(ii) If $G$ satisfies $CD(0,\infty)$ at $x$, then $\kappa(x,y) \geq 0$ for all $y\sim x$.\\
(iii) If $G$ does not satisfy $CD(0,\infty)$ at $x$, then there is a vertex $y\sim x$ for which $\kappa(x,y)\leq 0$.  If there is a vertex $y\sim x$ for which $\kappa(x,y) < 0$, then $G$ does not satisfy $CD(0,\infty)$ at $x$. \end{corollary}

\subsection
{Examples}
We give examples of graphs with no subgraph isomorphic to $K_3$ or $K_{2,3}$, categorized by which subhypothesis of Theorems 3.1 and 3.2 they satisfy.  Most of these graphs are vertex-transitive, so the curvature is identical at every vertex.

As before, let $v$ be a neighbor of $x$ that is not linked to the largest number of other neighbors of $x$, and let $N$ be that number.
\\
Common graphs with $N=0$, satisfying hypothesis (i) include \begin{itemize}
\item The hypercube $\Omega_d$.
\end{itemize}

Common graphs with $N = 1$, satisfying hypothesis (ii) include

\begin{itemize}
\item The square lattice $\Z^n$ where $n\geq 1$.
\item The cyclic graph $C_k$ if $k\geq 5$.
\item Any product of the above graphs, or product of these graphs with a graph from the previous list.
\end{itemize}

Common graphs with $N > 1$, satisfying hypothesis (iii) include

\begin{itemize}
\item The infinite $d$-regular tree if $d\geq 3$.
\item Any $d$-regular graph with girth $\geq 5$ if $d\geq 3$.
\item The graph of triangulations of an $n$-gon ($n\geq 6$), with edges representing the action of flipping one interior arc of the polygon.
\item $S_n$ with edges corresponding to adjacent transpositions.
\end{itemize}

\subsubsection{Interchange Process}
Given an underlying graph $H = ([n],F)$, the interchange process labels the vertices of $H$ and at each step, we are allowed to exchange the labels of a pair of adjacent vertices.  Let $G$ be the graph of possible states with an edge between two states if we can move from one state to the other in a single step.  In other words, $G$ is the Cayley graph of the subgroup of $S_n$ with generating set $A = \{(i,j):\{i,j\}\in F\}$.

$G$ is always $K_3$-free, and will be $K_{2,3}$-free if and only if $H$ is triangle-free.  If $x\in V(G)$ and $a,b\in A$, $ax\approx bx$ iff $a$ and $b$ are not incident to the same vertex as edges of $H$.  Because of this, we can easily determine the value of $N$ for the interchange process in terms of $H$.

\begin{itemize}

\item If $H$ is a matching, $G$ satisfies hypothesis (i). 

\item If $H$ is the disjoint union of paths with length $\leq 2$ (at least one of which has length equal to $2$), $G$ satisfies hypothesis (ii).

\item If $H$ has any vertex of degree $\geq 3$ or any path of length $\geq 3$, then $G$ satisfies hypothesis (iii). 
\end{itemize}

\section{Triangle-free regular graphs.}

In this section we compare the curvature-dimension inequality to the Ollivier curvature for regular triangle-free graphs.  Unlike Section 3, we allow the graph to have $K_{2,3}$ as a subgraph.  In this more general case our results are less complete.  In both notions of discrete curvature we give characterizations of classes of graphs with positive curvature, but it remains open to give a full characterization of graphs with positive curvature.

Let $G = (V,E)$ be a triangle-free $d$-regular graph.

\begin{definition}Let $u,w$ be two neighbors of $x$.  The \textit{linkage} of $u$ and $w$ is calculated by summing over all vertices $z\neq x$ for which $u\sim z\sim w$: \begin{align*}
l(u,w) = \sum_z \frac{1}{|\{y: x\sim y\sim z\}|} 
\end{align*}
\end{definition}

\begin{remark} As discussed in Section 2, if $G$ contains no subgraph isomorphic to $K_{2,3}$, then there can be at most one vertex $z$ with $z\neq x, u\sim z\sim w$, which will have $\{y:x\sim y\sum z\} = \{u,w\}$.  In this case $l(u,w) = 1/2$ and (as before) we say that $u$ and $w$ are linked.   
\end{remark}

\begin{theorem}
Let $G$ be a triangle-free graph, and $x\in V(G)$.  If for every pair $u,w$ of neighbors of $x$, $l(u,w)\geq \frac{1}{2}$, then $G$ satisfies $CD(\rho,\infty)$ at $x$ iff $\rho \leq 2$.\\
\end{theorem}

\begin{proof}
It is known that any triangle-free graph $G$ fails $CD(\rho,\infty)$ if $\rho > 2$ \cite{KKRT15}.

Recall that if $f$ is the minimizer of $\Gamma_2f(x)/\Gamma f(x)$ and $d(x,z) = 2$, \begin{align*}f(z) = \frac{2}{\# y: x\sim y\sim z}\sum_{y:x\sim y\sim z} f(y).\end{align*}

With this minimization (and the assumption of a triangle-free graph), straightforward algebraic manipulation reveals a form for $\Gamma_2 f$:
\begin{align}\label{gamma_linkage}2\Gamma_2 f(x) = \sum_{y\sim x} (3-d)f^2(y) + \sum_{w,y\sim x}2f(y) f(w) + \sum_{w,y\sim x} 2l(w,y) \parens{f(w)-f(y)}^2
\end{align}

Because $l\geq 1/2$ we can bound this equation:

\begin{align*}2\Gamma_2 f(x) &\geq \sum_{y\sim x} (3-d)f^2(y) + \sum_{w,y\sim x}2f(y) f(w) + \sum_{w,y\sim x} \parens{f(w)-f(y)}^2\\ &= \sum_{y\sim x} (3-d)f^2(y) + \sum_{w,y\sim x}2f(y) f(w) +\sum_{y\sim x} (d-1)f^2(y)- \sum_{w,y\sim x}2f(y)f(w)\\ &= 2\sum_{y\sim x}f^2(x) = 4\Gamma f(x),
\end{align*}
and therefore $CD(2,\infty)$ is satisfied.
\end{proof}

We prove a result regarding the Ollivier curvature of a related class of graphs:

\begin{theorem}
Let $G$ be a triangle-free graph. If every pair of adjacent edges of $G$ are contained within exactly one maximal complete bipartite induced subgraph whose parts have equal size, then $\kappa(x,y) = \frac{1}{d}$ for all pairs of neighbors $x,y$. 
\end{theorem}

\begin{proof}
Let $x,y$ be a pair of neighboring vertices in $G$.  The neighbors $w$ of $x$ with $w\neq y$ are partitioned into $S_1,...,S_k$ according to which maximal complete bipartite subgraph contains $\{xy\}$ and $\{xw\}$.  Similarly neighbors of $y$ (other than $x$) are partitioned into $T_1,...,T_k$, so that $\{x\}\cup T_i,\{y\}\cup S_i$ are the equally sized parts of a maximal complete bipartite graph.

To minimize $W_1(\mu(x),\mu(y))$ requires shifting a mass of size $\frac{|S_i|}{2d}$ from vertices of $S_i$ to those in $T_i$ and a mass of size $\frac{1}{2}-\frac{1}{2d}$ from $x$ to $y$ at a total cost of $\frac{1}{2}-\frac{1}{2d}+\frac{1}{2d}\sum_{i}|S_i| = 1 - \frac{2}{2d}$, this proves the theorem.
\end{proof}

Observe that under the hypothesis of Theorem 4.2, a pair of vertices $y,w$ that are neighbors of $x$ will be contained within a copy of $K_{m,m}$ for some value $m\geq 2$.  In this case, $l(w,y)\geq (m-1)\frac{1}{m}\geq 1/2$, so the hypothesis of Theorem 4.1 is also satisfied, and such a graph will have positive curvature both in terms of the $CD$ inequality and Ollivier curvature.

Examples of common graphs:\begin{itemize}

\item The hypercube $\Omega_n$, where neighbors $w$ and $y$ of $x$ are linked by a unique vertex $z$, so that $x,z,w,y$ are the vertices of a copy of $K_{2,2}$, and $l(x,y) = 1/2$.  $\Omega_n$ satisfies $CD(\rho,\infty)$ iff $\rho \leq 2$ and has $\kappa = 1/n$.

\item $G = K_{n,n}$ for $n\geq 2$, which has $l(y,w) = \frac{n-1}{n}$.  $K_{n,n}$ satisfies $CD(\rho,\infty)$ iff $\rho \leq 2$ and has $\kappa = 1/n$.

\item The Cayley graph of $S_n$ generated by all interchanges $(ij):i,j\in [n]$:

If $y = (ij)x$ and $w=(ik)x$, then $x,w,y$ are contained within the copy of $K_{3,3}$ that also includes $(jk)x,(ijk)x,(jik)x$.  In this case $l(w,y) = 2/3$.

On the other hand if $y = (ij)x$ and $w = (kl)x$, those vertices are contained within a square that includes $x$ and $(ij)(kl)x$, and $l(w,y) = 1/2$.

$G$ satisfies $K_{n,n}$ satisfies $CD(\rho,\infty)$ iff $\rho \leq 2$ and has $\kappa = 1/\binom{n}{2}$

\item If $X = ([n],E)$ is a graph, the interchange process on $X$ will always give an $|E|$-regular triangle-free graph.  If $X$ is a union of disjoint cliques, it is simple to see that $G$ satisfies the hypotheses of Theorems 4.1 and 4.2.  Indeed many of the previous examples are of this type: $\Omega_d$ corresponds to $X$ being a perfect matching on $2d$ vertices, and $S_n$ corresponds to $X = K_n$.

\item An $(n,d,k)$-incidence graph is a $d$-regular bipartite graph with partite sets $A_1,A_2$ of $n$ elements each so that each pair of elements in $A_i$ share $k$ common neighbors in $A_{1-i}$.  In order that such a graph exists, $(n-1)k = d(d-1)$ must count the number of $2$-paths starting (and not ending) at $x$.

Without loss of generality $x\in A_1$ and $x$ has neighbors $y,w\in A_2$, then $y$ and $w$ share $k-1$ other neighbors, each of which is adjacent to $k$ neighbors of $x$.  $l(w,y) = \frac{k-1}{k}$.  If $k\geq 2$, an $(n,d,k)$ incidence graph satisfies $CD(\rho,\infty)$ iff $\rho \leq 2$.

\end{itemize}

\section{Other Results}

\subsection{zig-zag product}
The zig-zag product is defined \cite{RVW} for graphs $G_1,G_2$, where 
\begin{itemize}
\item $G_1$ is regular with degree $d = |V_2|$, and for each $a\in V(G_1)$ the incident edges are indexed by $V_2$ - so that we can write $a[x]$ for the unique neighbor of $a$ that shares an edge labelled $x$ with $a$.
\item $G_2$ is $D$-regular.
\item The vertex set of the zig-zag product is $G_1\times G_2$.
\item If $x\sim y\sim z$ is a $2$-walk in $G_2$, then $(a,x)\sim (a[y],z)$.
\end{itemize}

The zig-zag product is a $D^2$-regular graph that inherits its expansion properties from $G_1$, which may have much larger degree.  For this reason the zig-zag product is useful in generating expander graphs of bounded degree.  The question arises whether the zig-zag product inherits the curvature properties from $G_1$ or $G_2$.  In general this is not the case, we give an example of such a graph:

If $G_1$ and $G_2$ are both abelian Cayley graphs, it is not in general true that the zig-zag product will have non-negative curavture: as a simple example if $G_1 = \Omega_d$ and $G_2 = Z_d$ with vertices $1\dots, d$ labelled in order around the cycle, and $x\in V_1$ with $y = x[2]$, $z = x[d]$, then $(x,1)$ has four neighbors - $(y,1), (y,3), (z,1), (z,d-1)$.\\$(y,1)$ has neighbors $(x,1),(x,3),(y[d],1),(y[d],d-1)$.\\
$(y,3)$ has neighbors $(x,1),(x,3),(y[4],3),(y[4],5)$.\\
$(z,1)$ has neighbors $(x,1),(x,d-1),(z[2],1),(z[2],3)$.\\
$(z,d-1)$ has neighbors $(x,1), (x,d-1),(z[d-2],d-1),(z[d-2],d-3)$.

$G$ is $K_3$ and $K_{2,3}$-free, we examine which pairs of neighbors of $(x,1)$ are not linked.  We see that $(y,3)\not\approx (z,1), (y,3)\not\approx (z,d-1), (y,1)\not\approx (z,d-1)$.  As such, $N=2$, and $G$ satisfies hypothesis (iii) of Theorems 3.1-2.

As such, $G$ fails $CD(0,\infty)$, and simple calculation gives $\kappa((x,1),(y,3)) = -\tfrac{1}{4}$.

\subsection{Diameter bounds}
It is well-known that a positive lower bound on Ollivier curvature gives an upper bound on the diameter of the graph, according to the following argument developed from \cite{Oll07}.

Assume for every pair of neighboring vertices $x,y$, $\kappa(x,y)\geq \kappa*>0$.  Let $x_0,\dots,x_l$ be a geodesic path, then \begin{align*}W_1(\mu_{x_0},\mu_{x_l})\leq \sum_{i=1}^l W_1(\mu_{x_{i-1}},\mu_{x_i}) \leq l(1-\kappa*),\end{align*} but, considering the $1$-Lipschitz function $f(y) = d(x_0,y)$, \begin{align*}W_1(\mu_{x_0},\mu_{x_l})\geq \int f d\mu_{x_l}-\int fd\mu_{x_0}\geq (l-\tfrac{1}{2})-(\tfrac{1}{2}) = l-1.\end{align*}

Thus $l(1-\kappa*)\geq l-1$, and $1\geq l\kappa*$; as the diameter is achieved on a geodesic path, its length $D$ is bounded above by $D\leq \frac{1}{\kappa*}$.

\begin{corollary}If $G$ is a $d$-regular graph with positive curvature, then the diameter is bounded by $D\leq 2d$.\end{corollary}

For any $x,z\in V$, $\mu_x(z)$ is an integer multiple of $\tfrac{1}{2d}$.  There is an optimal solution to the minimum-transport problem between $\mu_x$ and $\mu_y$ with the property that a multiple of $\tfrac{1}{2d}$ is transported along each path used - as each path has integer length, $W_1(\mu_x,\mu_y)$ is a multiple of $\tfrac{1}{2d}$.  If $x,y$ are neighbors and $\kappa(x,y) > 0$, then it must be $\kappa(x,y) \geq \tfrac{1}{2d}$.  If $x,y$ are not neighbors, then for any $z\sim x$ along the shortest $x-y$ path, $\kappa(x,y)\geq \kappa(x,z)\geq \tfrac{1}{2d}$.  Thus $\kappa* \geq \tfrac{1}{2d}$, and $D\leq 1/\kappa* \leq 2d$.

\begin{remark}
If $d$ is instead an upper bound on the degree of the vertex, then for any $x,y,z,w$, $\mu_x(z)$ and $\mu_y(w)$ are both multiples of $\frac{1}{2d_xd_y}$, and the same argument follows with $\kappa(x,y)\geq \frac{1}{2d_xd_y} \geq \frac{1}{2d^2-2d}$ and thus $D\leq 2d^2-2d$.
\end{remark}

\begin{remark}
This argument depends in the specifics on the definition of $\mu_x$ which is not universal, but for any of the other common definitions I am aware of a similar argument will follow.
\end{remark}

\begin{corollary} There is no infinite family of bounded-degree graphs with a lower bound on curvature of $\kappa > 0$.
\end{corollary}

It is known that a planar graph with bounded degree on $n$ vertices has spectral gap $\lambda = O(1/n)$\cite{ST07}.  The previous theorem tells us that among planar graphs with bounded degree, there is no infinite family that makes the bound $\kappa = O(1/n)$ tight, indeed, $\kappa \leq 0$ for all but finitely many bounded degree graphs.

In the case of graphs without bounded degree, it is true that $\kappa = 1/n$ for the star graph with $n$ leaves, which is planar.

\section*{Acknowledgements} The author thanks Max Fathi and Prasad Tetali for valuable discussions and examples of graphs relevant to this topic.

\bibliographystyle{plainnat}
\bibliography{ricci.bib}

\begin{thebibliography}{9}
\providecommand{\natexlab}[1]{#1}
\providecommand{\url}[1]{\texttt{#1}}
\expandafter\ifx\csname urlstyle\endcsname\relax
  \providecommand{\doi}[1]{doi: #1}\else
  \providecommand{\doi}{doi: \begingroup \urlstyle{rm}\Url}\fi

\bibitem[Bakry and {\'E}mery(1985)]{BE85}
D.~Bakry and M.~{\'E}mery.
\newblock Diffusions hypercontractives.
\newblock In \emph{S\'eminaire de probabilit\'es, {XIX}, 1983/84}, volume 1123
  of \emph{Lecture Notes in Math.}, pages 177--206. Springer, Berlin, 1985.
\newblock \doi{10.1007/BFb0075847}.
\newblock URL \url{http://dx.doi.org/10.1007/BFb0075847}.

\bibitem[Benjamini and Ellis()]{BE}
I.~Benjamini and D.~Ellis.
\newblock On the structure of graphs which are locally indistinguishable from a
  lattice.
\newblock URL \url{http://arxiv.org/abs/1409.7587}.

\bibitem[Chung et~al.(2014)Chung, Lin, and Yau]{CLY14}
F.~Chung, Y.~Lin, and S.-T. Yau.
\newblock Harnack inequalities for graphs with non-negative {R}icci curvature.
\newblock \emph{J. Math. Anal. Appl.}, 415\penalty0 (1):\penalty0 25--32, 2014.
\newblock ISSN 0022-247X.
\newblock \doi{10.1016/j.jmaa.2014.01.044}.
\newblock URL \url{http://dx.doi.org/10.1016/j.jmaa.2014.01.044}.

\bibitem[Klartag et~al.()Klartag, Kozma, Ralli, and Tetali]{KKRT15}
B.~Klartag, G.~Kozma, P.~Ralli, and P.~Tetali.
\newblock Discrete curvature and abelian groups.
\newblock URL \url{http://arxiv.org/pdf/1501.00516}.

\bibitem[Ollivier(2007)]{Oll07}
Y.~Ollivier.
\newblock Ricci curvature of metric spaces.
\newblock \emph{C. R. Math. Acad. Sci. Paris}, 345\penalty0 (11):\penalty0
  643--646, 2007.
\newblock ISSN 1631-073X.
\newblock \doi{10.1016/j.crma.2007.10.041}.
\newblock URL \url{http://dx.doi.org/10.1016/j.crma.2007.10.041}.

\bibitem[Ollivier(2009)]{Oll09}
Y.~Ollivier.
\newblock Ricci curvature of {M}arkov chains on metric spaces.
\newblock \emph{J. Funct. Anal.}, 256\penalty0 (3):\penalty0 810--864, 2009.
\newblock ISSN 0022-1236.
\newblock \doi{10.1016/j.jfa.2008.11.001}.
\newblock URL \url{http://dx.doi.org/10.1016/j.jfa.2008.11.001}.

\bibitem[Reingold et~al.(2000)Reingold, Vadhan, and Wigderson]{RVW}
O.~Reingold, S.~Vadhan, and A.~Wigderson.
\newblock Entropy waves, the zig-zag graph product, and new constant-degree
  expanders and extractors (extended abstract).
\newblock pages 3--13, 2000.
\newblock \doi{10.1109/SFCS.2000.892006}.
\newblock URL \url{http://dx.doi.org/10.1109/SFCS.2000.892006}.

\bibitem[Sammer(2005)]{Sam05}
M.D. Sammer.
\newblock {Aspects of mass transportation in discrete concentration
  inequalities}.
\newblock \emph{PhD thesis, Georgia Institute of Technology}, 2005.

\bibitem[Spielman and Teng(2007)]{ST07}
D.~Spielman and S.-H. Teng.
\newblock Spectral partitioning works: planar graphs and finite element meshes.
\newblock \emph{Linear Algebra Appl.}, 421\penalty0 (2-3):\penalty0 284--305,
  2007.
\newblock ISSN 0024-3795.
\newblock \doi{10.1016/j.laa.2006.07.020}.
\newblock URL \url{http://dx.doi.org/10.1016/j.laa.2006.07.020}.

\end{thebibliography}

\end{document}